\documentclass[12pt,a4paper]{article}
\usepackage{graphicx,amsmath,latexsym,amssymb,amsthm}
\usepackage{tikz}
\usetikzlibrary{arrows}
\usepackage[numbers]{natbib}

\newtheorem{thm}{Theorem}
\newtheorem{lem}[thm]{Lemma}
\newtheorem{cor}[thm]{Corollary}

\theoremstyle{definition}
\newtheorem{de}{Definition}
\newtheorem{ex}{Example}
\theoremstyle{remark}

\author{Penying Rochanakul}
\title{$k$-Geometric Mean Graphs}

\begin{document}

\maketitle

\begin{abstract}
A finite, simple and undirected graph $G = (V, E)$ with $p$ vertices and $q$ edges is said to be a \emph{$k$-geometric mean graph} for a positive integer $k$ if there is an injection $\psi :V(G)\to \{k,k+1,\dots,k+q\}$ such that, when each edge $uv\in E(G)$ is assigned the label $\lfloor\sqrt{\psi(u)\psi(v)}\rfloor$ or $\lceil\sqrt{\psi(u)\psi(v)}\rceil$, the resulting edge label set is $\{k,k+1,...,k+q-1\}$ and $\psi$ is called a \emph{$k$-geometric mean labeling} of $G$. The special case $k=1$, a $1$-geometric mean labeling is called a geometric mean labeling, and a $1$-geometric mean graph is called a geometric mean graph. 

In this paper, we present new classes of geometric mean graphs. Then we introduce $k$-geometric mean labeling and prove some classes of graphs are $k$-geometric mean. We also study some classes of finite join of graphs that admit geometric mean labeling.
\end{abstract}

\section{Introduction}
In this paper, we let $G$ be a finite, simple and undirected graph with vertex set $V(G)$ and edge set $E(G)$, where $|V(G)|=p$ and $|E(G)|=q$. 

Geometric mean labeling was first introduced by \citeauthor{som1} (2011) in  \cite{som1}. Many classes of geometric mean graphs are studied in \cite{som1}, \cite{som2}, \cite{som4}, \cite{som5}, \cite{som3}, and \cite{som6}. \citeauthor{durai} give a nice motivation and introduce the concept of $\mathcal{F}-$geometric mean labeling, where only the flooring function is used, in \cite{durai}. We here investigate further into new classes of geometric mean graphs and a finite join of them.

\begin{de}Let $G = (V, E)$ be a graph with $p$ vertices and $q$ edges. 

The graph $G$ is said to be a \emph{geometric mean graph} if there is an injection $\psi :V(G)\to \{1,2,\dots,q+1\}$ such that, when each edge $uv\in E(G)$ is assigned the label $\lfloor\sqrt{\psi(u)\psi(v)}\rfloor$ or $\lceil\sqrt{\psi(u)\psi(v)}\rceil$, the resulting edge label set is $\{1,2,...,q\}$ and $\psi$ is called a \emph{geometric mean labeling} of $G$. 
\end{de}

Here $P_n, C_n, K_n$, and $S_n$ are the path on $n$ vertices, cycle on $n$ vertices, complete graph on $n$ vertices, and star graph on $n+1$ vertices, respectively.

\begin{ex}
A geometric mean labeling of $C_5$ is given below.
\end{ex}
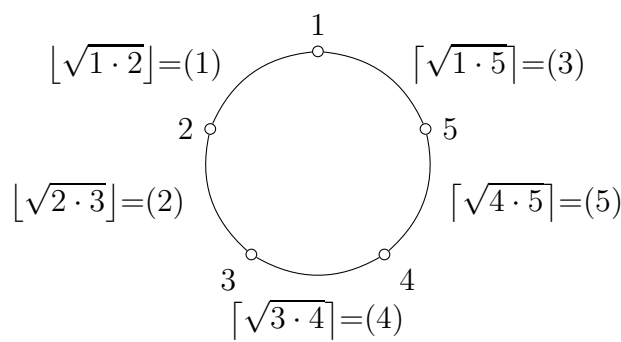
\begin{center}
\begin{tikzpicture}[scale=.5]
                            inner sep=0pt]
    \draw (0,0) node[draw,circle,fill=white,minimum size=4pt,
                                inner sep=0pt] (A) [label=above:$1$] {}
         ++(216:3.5cm) node[draw,circle,fill=white,minimum size=4pt,
                                     inner sep=0pt] (B) [label=left:$2$] {}
         ++(288:3.5cm) node[draw,circle,fill=white,minimum size=4pt,
                                     inner sep=0pt] (C) [label=below left:$3$] {}
         ++(0:3.5cm) node[draw,circle,fill=white,minimum size=4pt,
                                     inner sep=0pt] (D) [label=below right:$4$] {}
         ++(72:3.5cm) node[draw,circle,fill=white,minimum size=4pt,
                                     inner sep=0pt] (E) [label=right:$5$] {}
         ++(120:2cm) node  (1) [label=right:{$\left\lceil\sqrt{1\cdot 5}\right\rceil$=(3)}] {}
		++(285:3.8cm) node  (2) [label=right:{$\left\lceil\sqrt{4\cdot 5}\right\rceil$=(5)}] {}
		++(215:3.5cm) node  (3) [label=below:{$\left\lceil\sqrt{3\cdot 4}\right\rceil$=(4)}] {}
		++(145:3.5cm) node  (4) [label=left:{$\left\lfloor\sqrt{2\cdot 3}\right\rfloor$=(2)}] {}
		++(75:3.8cm) node  (5) [label=left:{$\left\lfloor\sqrt{1\cdot 2}\right\rfloor$=(1)}] {}
        ;

  \path (A) edge [bend right]   (B)
        (B)	edge [bend right]   (C)
        (C) edge [bend right]   (D)
        (D)	edge [bend right]   (E)
        (E) edge [bend right]   (A);
\end{tikzpicture}
\end{center}

\begin{de}The \emph{union} $G_1\cup G_2$ of two graphs $G_1$ and $G_2$ is the graph $G$ with $V(G)=V(G_1)\cup V(G_2)$ and $E(G)=E(G_1)\cup E(G_2)$.
\end{de}

\begin{de}The \emph{join} $G_1 + G_2$ of two graphs $G_1$ and $G_2$ with disjoint set of vertices is the union of two graphs $G_1$ and $G_2$.
\end{de}

\begin{de}The \emph{corona} $G_1 \odot G_2$ of two graphs $G_1$ with $p$ vertices and $G_2$ is the graph $G$ obtained by taking one copy of $G_1$ and $p$ copies of $G_2$ and then joining the $i^{th}$ vertex of $G_1$ to every vertices in the $i^{th}$ copy of $G_2$.
\end{de}

\begin{thm}[\cite{som1}] 
Any path  is a geometric mean graph.
\end{thm}

\begin{thm}[\cite{som1}]
Any cycle is a geometric mean graph.
\end{thm}

\begin{de}A \emph{comb} is the graph $P_n \odot K_1$.
\end{de}

\begin{thm}[\cite{som1}]
Any comb is a geometric mean graph.
\end{thm}

\begin{de}A \emph{crown} is the graph $C_n \odot K_1$.
\end{de}

\begin{thm}[\cite{som2}]
Any crown is a geometric mean graph.
\end{thm}

\begin{de}A \emph{triangular snake} $T_n$ is obtained from replacing each edge in $P_n$ with $C_3$.
\end{de}

\begin{thm}[\cite{som3}]
Any triangular snake $T_n$ is a geometric mean graph.
\end{thm}

\begin{de}A \emph{quadrilateral snake} $Q_n$ is obtained from replacing each edge in $P_n$ with $C_4$.
\end{de}

\begin{thm}[\cite{som3}]
Any quadrilateral snake $Q_n$ is a geometric mean graph.
\end{thm}

Somasundram et al. provide some results on the join of graphs in \cite{som1}, \cite{som2} and \cite{som3} as follows.

\begin{thm}[\cite{som2}]
$C_m + P_n$ is a geometric mean graph.
\end{thm}

\begin{thm}[\cite{som2}]
$C_m + C_n$ is a geometric mean graph.
\end{thm}

\begin{thm}[\cite{som2}]
For any $n\in \mathbb{N}$, $n$ disjoint copy of $C_3$ denoted by $nC_3$ is a geometric mean graph.
\end{thm}

\begin{thm}[\cite{som2}]
For any $m,n\in \mathbb{N}$, $nC_3 + P_m$ is a geometric mean graph.
\end{thm}

\begin{thm}[\cite{som2}]
For any $m,n\in \mathbb{N}$, $nC_3 + C_m$ is a geometric mean graph.
\end{thm}

\begin{thm}[\cite{som3}]
$C_m + T_n$ is a geometric mean graph.
\end{thm}

\begin{thm}[\cite{som3}]
$(C_m\odot K_1) + T_n$ is a geometric mean graph.
\end{thm}

\begin{thm}[\cite{som3}]
$C_m + Q_n$ is a geometric mean graph.
\end{thm}

\begin{thm}[\cite{som3}]
$(C_m\odot K_1) + Q_n$ is a geometric mean graph.
\end{thm}


Here we give some new results on geometric mean graphs and establish a general result on finite join of paths, cycles, combs, crowns, triangle snakes, and quadrilateral snakes.  

\section{Main results}
\subsection{Stars graphs}
\begin{lem}\label{notgeo}
Let $G$ be a graph. If $|V(G)|>|E(G)|=1$, then $G$ is not a geometric mean graph.
\end{lem}
\begin{proof}
If $|V(G)|>|E(G)|=1$, then the injective vertex labeling does not exist and hence $G$ cannot be a geometric mean graph.
\end{proof}
\begin{thm}\label{comnongeo}
The join of any two graphs from the set of trees (including paths, combs, and stars), triangle snakes, and quadrilateral snakes, is not a geometric mean graph.
\end{thm}
\begin{proof} Let $G$ be a join of graphs $G_1$ and $G_2$, which are from the set of trees, triangle snakes, and quadrilateral snakes. Then\begin{align*}
|V(G)|&=|V(G)_1|+|V(G_2)|\\&\geq \left(|E(G)_1|+1\right)+\left(|V(G_2)|+1\right)\\&> |E(G)_1|+|E(G_2)|+1\\&=|E(G)|+1. 
\end{align*} 
By Lemma \ref{notgeo}, $G$ is not a geometric mean graph.
\end{proof}
\begin{thm}
If the star $S_n$ is a geometric mean graph then $n\leq 7$.
\end{thm}
\begin{proof}
A star $S_n$ has exactly $n+1$ vertices and $n$ edges. Hence all labels in  $\{1,2,\dots,n+1\}$ are used on vertices labeling. 

Assume that there exists a geometric mean labeling of $S_n$. Let the only vertex of degree greater than 1 be labelled by $k\in\{1,2,\dots,n+1\}$. Therefore the lowest possible edge label is $\lfloor\sqrt{1\cdot k}\rfloor$ and the highest posible edge label is $\lceil\sqrt{k\cdot (n+1)}\rceil$

Since the graph need at least $n$ edge labels. It is nessesary that \begin{equation}\lceil\sqrt{k\cdot (n+1)}\rceil-\lfloor\sqrt{1\cdot k}\rfloor+1\geq n.\label{firstcon}\tag{$\star$}\end{equation}
Since $\lfloor\sqrt{1\cdot k}\rfloor\geq 1$, to achieve (\ref{firstcon}), we need $\lceil\sqrt{k\cdot (n+1)}\rceil$ to be greater than $n$. Thus $k$ is required to be in $\{n-2,n-1,n,n+1\}$. 

\noindent\textbf{Case $k=n-2$ or $n-1$;}\\
Consider
\begin{align*}
n&\leq\lceil\sqrt{k\cdot (n+1)}\rceil-\lfloor\sqrt{1\cdot k}\rfloor+1\\
&\leq n-\lfloor\sqrt{1\cdot (n-2)}\rfloor+1.
\end{align*}
Hence $\lfloor\sqrt{1\cdot (n-2)}\rfloor=1$. Which implies $n\in\{1,2,\dots,5\}$.

\noindent\textbf{Case $k=n$ or $n+1$;}\\
Consider
\begin{align*}
n&\leq\lceil\sqrt{k\cdot (n+1)}\rceil-\lfloor\sqrt{1\cdot k}\rfloor+1\\
&\leq(n+1)-\lfloor\sqrt{1\cdot n}\rfloor+1.
\end{align*}
Hence $\lfloor\sqrt{1\cdot n}\rfloor=2$ or $1$. Which implies $n\in\{1,2,\dots,8\}$. However, without too much effort, one can see that when $n=8$, both $k=8$ and $k=9$ do not give a geometric mean labeling. 
\end{proof}

\begin{ex}
Here are examples of geometric mean labeling of $S_n$ for $n=1$ to $7$.
\end{ex}
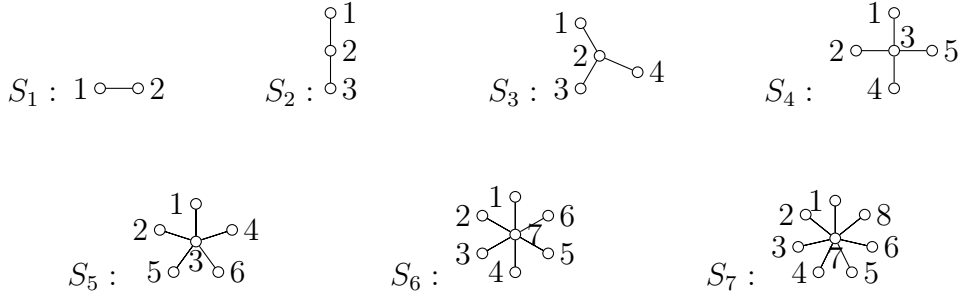
\begin{center}
$S_1:$ \begin{tikzpicture}[scale=.5]
    \tikzstyle{every node}=[draw,circle,fill=white,minimum size=4pt,
                            inner sep=0pt]

    \draw (0,0) node (0) [label=left:$1$] {}
 -- ++(0:1cm) node (0) [label=right:$2$] {};
  
\end{tikzpicture}\qquad\quad
$S_2:$ \begin{tikzpicture}[scale=.5]
    \tikzstyle{every node}=[draw,circle,fill=white,minimum size=4pt,
                            inner sep=0pt]

    \draw (0,0) node (0) [label=right:$1$] {}
        -- ++(270:1cm) node (1) [label=right:$2$] {}
        -- ++(270:1cm) node (2) [label=right:$3$] {};

\end{tikzpicture}\qquad\qquad
$S_3:$ \begin{tikzpicture}[scale=.5]
    \tikzstyle{every node}=[draw,circle,fill=white,minimum size=4pt,
                            inner sep=0pt]

    \draw (0,0) node (1) [label=left:$1$] {}
        -- ++(300:1cm) node (0) [label=left:$2$] {}
        -- ++(240:1cm) node (2) [label=left:$3$] {};
 	\node (3) at (1.5,-1.3) [label=right:$4$]{ };

    \draw (0) -- (3);
  
\end{tikzpicture}\qquad\quad
$S_4:$ \begin{tikzpicture}[scale=.5]
    \tikzstyle{every node}=[draw,circle,fill=white,minimum size=4pt,
                            inner sep=0pt]

    \draw (0,0) node (1) [label=above right:$3$] {};
    \draw (0,1) node (2) [label=left:$1$] {};
    \draw (0,-1) node (3) [label=left:$4$] {};
    \draw (-1,0) node (4) [label=left:$2$] {};
    \draw (1,0) node (5) [label=right:$5$] {};

    \draw (1) -- (2);
    \draw (1) -- (3);
    \draw (1) -- (4);
    \draw (1) -- (5);
  
\end{tikzpicture}\\\vskip1cm
$S_5:$ \begin{tikzpicture}[scale=.5]
    \tikzstyle{every node}=[draw,circle,fill=white,minimum size=4pt,
                            inner sep=0pt]

    \draw (0,0) node (1) [label=below:$3$] {}
        -- ++(90:1cm) node (2) [label=left:$1$] {}
        -- (1)
        -- ++(162:1cm) node (3) [label=left:$2$] {}
        -- (1)
        -- ++(18:1cm) node (4) [label=right:$4$] {}
        -- (1)
        -- ++(234:1cm) node (5) [label=left:$5$] {}
        -- (1)
        -- ++(306:1cm) node (6) [label=right:$6$] {}
        ;
        

  
\end{tikzpicture}\qquad\qquad
$S_6:$ \begin{tikzpicture}[scale=.5]
    \tikzstyle{every node}=[draw,circle,fill=white,minimum size=4pt,
                            inner sep=0pt]

        \draw (0,0) node (1) [label=right:$7$] {}
            -- ++(90:1cm) node (2) [label=left:$1$] {}
            -- (1)
            -- ++(150:1cm) node (3) [label=left:$2$] {}
            -- (1)
            -- ++(210:1cm) node (4) [label=left:$3$] {}
            -- (1)
            -- ++(270:1cm) node (5) [label=left:$4$] {}
            -- (1)
            -- ++(330:1cm) node (6) [label=right:$5$] {}
            -- (1)
			-- ++(30:1cm) node (7) [label=right:$6$] {}
            ;
  
\end{tikzpicture}\qquad\qquad
$S_7:$ \begin{tikzpicture}[scale=.5]
    \tikzstyle{every node}=[draw,circle,fill=white,minimum size=4pt,
                            inner sep=0pt]

    \draw (0,0) node (1) [label=below:$7$] {}
        -- ++(90:1cm) node (2) [label=left:$1$] {}
        -- (1)
        -- ++(39:1cm) node (3) [label=right:$8$] {}
        -- (1)
        -- ++(141:1cm) node (4) [label=left:$2$] {}
        -- (1)
        -- ++(193:1cm) node (5) [label=left:$3$] {}
        -- (1)
        -- ++(347:1cm) node (6) [label=right:$6$] {}
         -- (1)
                -- ++(244:1cm) node (7) [label=left:$4$] {}
                -- (1)
                -- ++(296:1cm) node (8) [label=right:$5$] {}
                ;


  
\end{tikzpicture}
\end{center}
\subsection{$k$-geometric mean graphs}
Here we define $k$-geometric mean labeling as a tool to proof our main theorem.

\begin{de}\label{kgeo}A graph $G=(V, E)$ with $p$ vertices and $q$ edges is said to be a \emph{$k$-geometric mean graph} ($k$ ia a positive integer) if there is an injection $\psi :V(G)\to \{k,k+1,\dots,k+q\}$ such that, when each edge $uv\in E(G)$ is assigned the label $\lfloor\sqrt{\psi(u)\psi(v)}\rfloor$ or $\lceil\sqrt{\psi(u)\psi(v)}\rceil$, the resulting edge label set is $\{k,k+1,...,k+q-1\}$ and $\psi$ is called a \emph{$k$-geometric mean labeling} of $G$. 
\end{de}

Note that from Definition \ref{kgeo}, a geometric mean labeling is a $1$-geometric mean labeling.

\begin{thm}\label{path}
Let $n,k\in \mathbb{N}$, $P_n$ is a $k$-geometric mean graph.
\end{thm}
\begin{proof}Let $P_n$ be the path $u_1u_2\dots u_n$.

Define a function $\psi:V(P_n)\rightarrow \{k,k+1,\dots,k+(n-1)=k+q\}$ by

\hspace{2cm}$\psi(u_i)=(k-1)+i,\qquad1\leq i \leq n$

For $i=1$ to $n-1$, assign the label $\lfloor\sqrt{\psi(u_i)\psi(u_{i+1})}\rfloor$ to edge $u_iu_{i+1}$.  Then the set of edge labels is $\{k,k+1,...,k+(n-1)-1=k+q-1\}$. Hence $P_n$ is a $k$-geometric mean graph.
\end{proof}
\begin{ex}
A $3$-geometric mean labeling of $P_5$ is given below.
\end{ex}
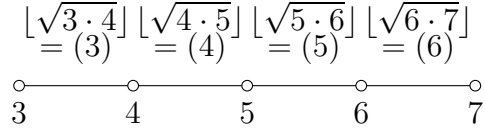
\begin{center}
\begin{tikzpicture}[scale=.5]
                            inner sep=0pt]
    \draw (0,0) node[draw,circle,fill=white,minimum size=4pt,
                                inner sep=0pt] (A) [label=below:$3$] {}
        -- ++(0:3cm) node[draw,circle,fill=white,minimum size=4pt,
                                     inner sep=0pt] (B) [label=below:$4$] {}
       --  ++(0:3cm) node[draw,circle,fill=white,minimum size=4pt,
                                     inner sep=0pt] (C) [label=below :$5$] {}
        -- ++(0:3cm) node[draw,circle,fill=white,minimum size=4pt,
                                     inner sep=0pt] (D) [label=below  :$6$] {}
       --  ++(0:3cm) node[draw,circle,fill=white,minimum size=4pt,
                                     inner sep=0pt] (E) [label=below:$7$] {}
       ++(180:1.5cm)  node (2) [label=above:{$=(6)$}] {}
       ++(180:3cm) node (3) [label=above:{$=(5)$}] {}
       ++(180:3cm) node (4) [label=above:{$=(4)$}] {}
       ++(180:3cm) node (5) [label=above:{$=(3)$}] {}
       ++(90:0.7cm)  node (2) [label=above:{$\lfloor\sqrt{3\cdot 4}\rfloor$}] {}
              ++(0:3cm) node (3) [label=above:{$\lfloor\sqrt{4\cdot 5}\rfloor$}] {}
              ++(0:3cm) node (4) [label=above:{$\lfloor\sqrt{5\cdot 6}\rfloor$}] {}
              ++(0:3cm) node (5) [label=above:{$\lfloor\sqrt{6\cdot 7}\rfloor$}] {}
        ;

\end{tikzpicture}
\end{center}

\begin{thm}\label{cycle}
Let $n,k\in \mathbb{N}$ and $n\geq 3$, $C_n$ is a $k$-geometric mean graph.
\end{thm}
\begin{proof}Let $C_n$ be the cycle $u_1u_2\dots u_nu_1$.

Define a function $\psi:V(C_n)\rightarrow \{k,k+1,\dots,k+n=k+q\}$ by

\hspace{2cm}$\psi(u_i)=(k-1)+i,\qquad1\leq i \leq n$

Let $h=\left\lceil\sqrt{k(k+n)}\right\rceil$. We have $k<h<k+n-1$. 

Assign the label $h$ to edge $u_1u_n$. For each edge $u_iu_{i+1}$ assign the label $\left\lfloor\sqrt{\psi(u_i)\psi(u_{i+1})}\right\rfloor$ when $1\leq i \leq h-k=1$, assign the label $\left\lceil\sqrt{\psi(u_i)\psi(u_{i+1})}\right\rceil$ otherwise. Then the set of edge labels is $\{k,k+1,...,k+n-1=k+q-1\}$. Hence $C_n$ is a $k$-geometric mean graph.
\end{proof}
\begin{ex}
A $4$-geometric mean labeling of $C_5$ is given below.
\end{ex}
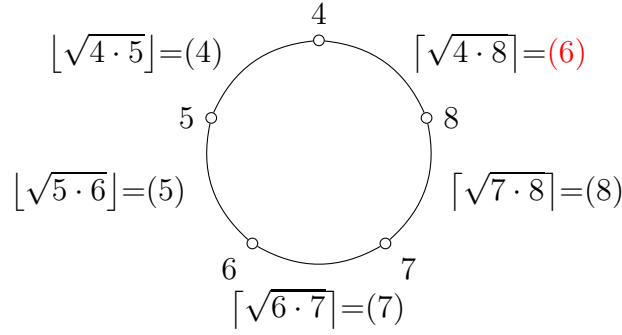
\begin{center}
\begin{tikzpicture}[scale=.5]
                            inner sep=0pt]
    \draw (0,0) node[draw,circle,fill=white,minimum size=4pt,
                                inner sep=0pt] (A) [label=above:$4$] {}
         ++(216:3.5cm) node[draw,circle,fill=white,minimum size=4pt,
                                     inner sep=0pt] (B) [label=left:$5$] {}
         ++(288:3.5cm) node[draw,circle,fill=white,minimum size=4pt,
                                     inner sep=0pt] (C) [label=below left:$6$] {}
         ++(0:3.5cm) node[draw,circle,fill=white,minimum size=4pt,
                                     inner sep=0pt] (D) [label=below right:$7$] {}
         ++(72:3.5cm) node[draw,circle,fill=white,minimum size=4pt,
                                     inner sep=0pt] (E) [label=right:$8$] {}
         ++(120:2cm) node  (1) [label=right:{$\left\lceil\sqrt{4\cdot 8}\right\rceil$=\color{red}(6)}] {}
		++(285:3.8cm) node  (2) [label=right:{$\left\lceil\sqrt{7\cdot 8}\right\rceil$=(8)}] {}
		++(215:3.5cm) node  (3) [label=below:{$\left\lceil\sqrt{6\cdot 7}\right\rceil$=(7)}] {}
		++(145:3.5cm) node  (4) [label=left:{$\left\lfloor\sqrt{5\cdot 6}\right\rfloor$=(5)}] {}
		++(75:3.8cm) node  (5) [label=left:{$\left\lfloor\sqrt{4\cdot 5}\right\rfloor$=(4)}] {}
        ;

  \path (A) edge [bend right]   (B)
        (B)	edge [bend right]   (C)
        (C) edge [bend right]   (D)
        (D)	edge [bend right]   (E)
        (E) edge [bend right]   (A);
\end{tikzpicture}
\end{center}

\begin{thm}\label{crown}
Let $n,k\in \mathbb{N}$ and $n\geq 3$, $C_n\odot K_1$ is a $k$-geometric mean graph.
\end{thm}
\begin{proof}Let $C_n$ be the cycle $u_1u_2\dots u_nu_1$ and let $v_i$ be the vertex adjacent to $u_i$ for $1\leq 1 \leq n$.

Define a function $\psi:V(C_n\odot K_1)\rightarrow \{k,k+1,\dots,k+2n=k+q\}$ by

\hspace{2cm}$\psi(u_i)=(k-1)+2i,\qquad1\leq i \leq n$

\hspace{2cm}$\psi(v_i)=(k-1)+2i-1,\qquad\qquad1\leq i \leq n$

Let $h=\left\lceil\sqrt{(k+1)(k+2n-1)}\right\rceil$. We have $k+1<h<k+2n-1$. 

Assign the label $h$ to edge $u_1u_n$. For each edge $uv$ in $C_n\odot K_1$ assign the label as follows\\
\textbf{Case 1: } $h=(k-1)+2j=k+2j-1$ for some $j \in \{2,3,\dots,n-1\}$
\begin{center}
\begin{tabular}{|c| c| c|}
\hline 
edge $uv$ & label & obtained labels \\\hline&&\\[-8pt]
$u_iu_{i+1},\quad1\leq i \leq j-1$&$\left\lfloor\sqrt{\psi(u)\psi(v)}\right\rfloor$&$\{k+1,k+3,\dots,k+2j-3\}$\\[5pt]\hline&&\\[-8pt]
$u_iu_{i+1},\quad j+1\leq i \leq n-1$&$\left\lceil\sqrt{\psi(u)\psi(v)}\right\rceil$&$\{k+2j,k+2j+2,\dots,k+2n-2\}$\\[5pt]\hline&&\\[-8pt]
$u_iv_i,\quad 1\leq i \leq j$&$\left\lfloor\sqrt{\psi(u)\psi(v)}\right\rfloor$&$\{k,k+2,\dots,k+2j-2\}$\\[5pt]\hline&&\\[-8pt]
$u_iv_i,\quad j+1\leq i \leq n$&$\left\lceil\sqrt{\psi(u)\psi(v)}\right\rceil$&$\{k+2j+1,k+2j+3,\dots,k+2n-1\}$\\[5pt]\hline
\end{tabular}
\end{center}
\vskip0.5cm
\noindent\textbf{Case 2: } $h=(k-1)+2j-1=k+2j-2$ for some $j \in \{2,3,\dots,n-1\}$
\begin{center}
\begin{tabular}{|c| c| c|}
\hline 
edge $uv$ & label & obtained labels \\\hline&&\\[-8pt]
$u_iu_{i+1},\quad1\leq i \leq j-1$&$\left\lfloor\sqrt{\psi(u)\psi(v)}\right\rfloor$&$\{k+1,k+3,\dots,k+2j-3\}$\\[5pt]\hline&&\\[-8pt]
$u_iu_{i+1},\quad j+1\leq i \leq n-1$&$\left\lceil\sqrt{\psi(u)\psi(v)}\right\rceil$&$\{k+2j,k+2j+2,\dots,k+2n-2\}$\\[5pt]\hline&&\\[-8pt]
$u_iv_i,\quad 1\leq i \leq j-1$&$\left\lfloor\sqrt{\psi(u)\psi(v)}\right\rfloor$&$\{k,k+2,\dots,k+2j-4\}$\\[5pt]\hline&&\\[-8pt]
$u_iv_i,\quad j\leq i \leq n$&$\left\lceil\sqrt{\psi(u)\psi(v)}\right\rceil$&$\{k+2j-1,k+2j+1,\dots,k+2n-1\}$\\[5pt]\hline
\end{tabular}
\end{center}
The sets of edge labels in both cases are $\{k,k+1,...,k+2n-1=k+q-1\}$. Hence $C_n\odot K_1$ is a $k$-geometric mean graph.
\end{proof}

Observe that the number of vertices and the number of edges are equal for cycles and crowns. Moreover, we can label vertices of cycles and crown without using the label $k+q$.

\begin{ex}
A $4$-geometric mean labeling  of $C_4\odot K_1$ is given below.
\end{ex}

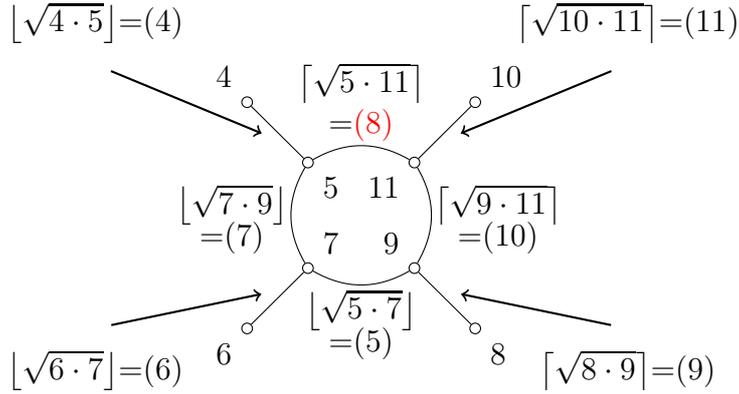
\begin{center}
\begin{tikzpicture}
                            inner sep=0pt]
    \draw (0,0) node[draw,circle,fill=white,minimum size=4pt,
                                inner sep=0pt] (v1) [label=above left:$4$] {}
         ++(270:3cm) node[draw,circle,fill=white,minimum size=4pt,
                                     inner sep=0pt] (v2) [label=below left:$6$] {}
         ++(0:3cm) node[draw,circle,fill=white,minimum size=4pt,
                                     inner sep=0pt] (v3) [label=below right:$8$] {}
         ++(90:3cm) node[draw,circle,fill=white,minimum size=4pt,
                                     inner sep=0pt] (v4) [label=above right:$10$] {}
        ;
    \draw (0.8,-0.8) node[draw,circle,fill=white,minimum size=4pt,
                                inner sep=0pt] (u1) [label=below right:$5$] {}
         ++(270:1.4cm) node[draw,circle,fill=white,minimum size=4pt,
                                     inner sep=0pt] (u2) [label=above right:$7$] {}
         ++(0:1.4cm) node[draw,circle,fill=white,minimum size=4pt,
                                     inner sep=0pt] (u3) [label=above left:$9$] {}
         ++(90:1.4cm) node[draw,circle,fill=white,minimum size=4pt,
                                     inner sep=0pt] (u4) [label=below left:$11$] {}
        ;
    	\draw (1.5,-0.8) node (11) [label=above:{=\color{red}(8)}] {};     
	\draw (1.5,-0.3) node (12) [label=above:{$\left\lceil\sqrt{5\cdot 11}\right\rceil$}] {};
    	\draw (1.5,-2.7) node (21) [label=below:{=(5)}] {};     
	\draw (1.5,-2.2) node (22) [label=below:{$\left\lfloor\sqrt{5\cdot 7}\right\rfloor$}] {};
    	\draw (-0.2,-1.3) node (31) [label=below:{=(7)}] {};     
	\draw (-0.2,-0.8) node (32) [label=below:{$\left\lfloor\sqrt{7\cdot 9}\right\rfloor$}] {};
    	\draw (3.3,-1.3) node (41) [label=below:{=(10)}] {};     
	\draw (3.3,-0.8) node (42) [label=below:{$\left\lceil\sqrt{9\cdot 11}\right\rceil$}] {};
	\draw (-2,0.5) node (52) [label=above:{$\left\lfloor\sqrt{4\cdot 5}\right\rfloor$=(4)}] {};
	\draw (5,0.5) node (62) [label=above:{$\left\lceil\sqrt{10\cdot 11}\right\rceil$=(11)}] {};
	\draw (-2,-3) node (72) [label=below:{$\left\lfloor\sqrt{6\cdot 7}\right\rfloor$=(6)}] {};
	\draw (5,-3) node (82) [label=below:{$\left\lceil\sqrt{8\cdot 9}\right\rceil$=(9)}] {};
    	\draw (0.4,-0.5) node (p1)  {};   
    	\draw (2.6,-0.5) node (p2)  {};   
    	\draw (0.4,-2.5) node (p3)  {};   
    	\draw (2.6,-2.5) node (p4)  {};    
  \path (u1) edge [bend right]   (u2)
        (u2)	edge [bend right]   (u3)
        (u3) edge [bend right]   (u4)
        (u4)	edge [bend right]   (u1)
	(52) edge [->,thick,
           shorten <=2pt,
           shorten >=2pt] (p1)
	(62) edge[->,thick,
           shorten <=2pt,
           shorten >=2pt] (p2)
	(72) edge [->,thick,
           shorten <=2pt,
           shorten >=2pt](p3)
	(82) edge[->,thick,
           shorten <=2pt,
           shorten >=2pt](p4)
	(u1) edge     (v1)
	(u2) edge     (v2)
	(u3) edge     (v3)
        (u4) edge     (v4);
\end{tikzpicture}
\end{center}

\begin{thm}\label{comb}
For any $n\in \mathbb{N}$, $P_n\odot K_1$ is a $k$-geometric mean graph.
\end{thm}
\begin{proof}Let $P_n$ be the path $u_1u_2\dots u_n$ and let $v_i$ be the vertex adjacent to $u_i$ for $1\leq 1 \leq n$.

Define a function $\psi:V(P_n\odot K_1)\rightarrow \{k,k+1,\dots,k+2n-1=k+q\}$ by

\hspace{2cm}$\psi(u_i)=(k-1)+2i-1,\qquad1\leq i \leq n$

\hspace{2cm}$\psi(v_i)=(k-1)+2i,\qquad\qquad1\geq i \geq n$

Assign the label $\left\lceil\sqrt{\psi(u_i)\psi(u_{i+1})}\right\rceil$ to edge $u_iu_{i+1}$, for $1\leq i \leq n-1$, and assign the label $\left\lfloor\sqrt{\psi(u_i)\psi(v_i)}\right\rfloor$ to edge $u_iv_i$, for $1\leq i \leq n$. Then the set of edge labels is $\{k,k+1,...,k+2n-2=k+q-1\}$. Hence $P_n\odot K_1$ is a $k$-geometric mean graph.
\end{proof}
\begin{ex}
A $3$-geometric mean labeling of $P_5\odot K_1$ is given below.
\end{ex}
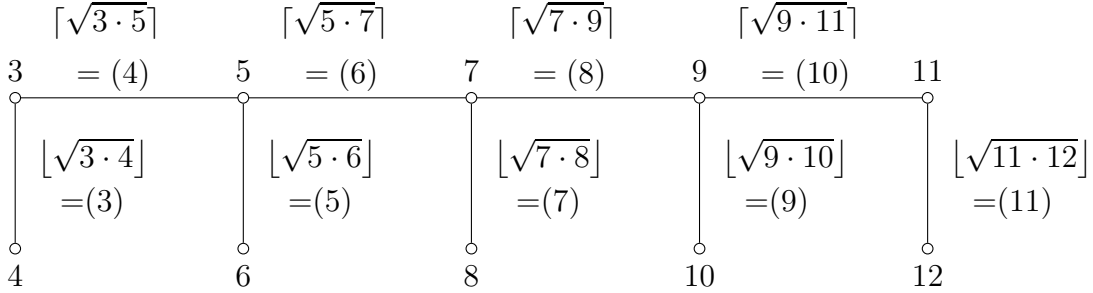
\begin{center}
\begin{tikzpicture}
                            inner sep=0pt]
    \draw (0,0) node[draw,circle,fill=white,minimum size=4pt,
                                inner sep=0pt] (A) [label=above:$3$] {}
        -- ++(0:3cm) node[draw,circle,fill=white,minimum size=4pt,
                                     inner sep=0pt] (B) [label=above:$5$] {}
       --  ++(0:3cm) node[draw,circle,fill=white,minimum size=4pt,
                                     inner sep=0pt] (C) [label=above :$7$] {}
        -- ++(0:3cm) node[draw,circle,fill=white,minimum size=4pt,
                                     inner sep=0pt] (D) [label=above  :$9$] {}
       --  ++(0:3cm) node[draw,circle,fill=white,minimum size=4pt,
                                     inner sep=0pt] (E) [label=above:$11$] {}
       ++(173:2.5cm)  node (2) [label=right:{$=(10)$}] {}
       ++(180:3cm) node (3) [label=right:{$=(8)$}] {}
       ++(180:3cm) node (4) [label=right:{$=(6)$}] {}
       ++(180:3cm) node (5) [label=right:{$=(4)$}] {}
       ++(115:0.8cm)  node (2) [label=right:{$\lceil\sqrt{3\cdot 5}\rceil$}] {}
       ++(0:3cm) node (3) [label=right:{$\lceil\sqrt{5\cdot 7}\rceil$}] {}
       ++(0:3cm) node (4) [label=right:{$\lceil\sqrt{7\cdot 9}\rceil$}] {}
       ++(0:3cm) node (5) [label=right:{$\lceil\sqrt{9\cdot 11}\rceil$}] {}
        ;
   \draw (0,-2) node[draw,circle,fill=white,minimum size=4pt,
                                   inner sep=0pt] (F) [label=below:$4$] {}
         ++(0:3cm) node[draw,circle,fill=white,minimum size=4pt,
                                            inner sep=0pt] (G) [label=below:$6$] {} 
          ++(0:3cm) node[draw,circle,fill=white,minimum size=4pt,
                                                     inner sep=0pt] (H) [label=below:$8$]{}  
          ++(0:3cm) node[draw,circle,fill=white,minimum size=4pt, inner sep=0pt](I)[label=below:$10$] {}                                                                                                  ++(0:3cm) node[draw,circle,fill=white,minimum size=4pt,
                                                      inner sep=0pt] (J) [label=below:$12$]{};
  \draw (12,-0.8) node (11) [label=right:{$\left\lfloor\sqrt{11\cdot 12}\right\rfloor$ }] {}
   ++(180:3cm)node(12)[label=right:{$\left\lfloor\sqrt{9\cdot 10}\right\rfloor$ }]{}                                               ++(180:3cm)node(13)[label=right:{$\left\lfloor\sqrt{7\cdot 8}\right\rfloor$ }]{}
   ++(180:3cm)node(14)[label=right:{$\left\lfloor\sqrt{5\cdot 6}\right\rfloor$ }]{}
   ++(180:3cm)node(15)[label=right:{$\left\lfloor\sqrt{3\cdot 4}\right\rfloor$ }]{};     
     \draw (12.3,-1.4) node (11) [label=right:{=(11)}] {}
      ++(180:3cm)node(12)[label=right:{=(9)}]{}                                               ++(180:3cm)node(13)[label=right:{=(7)}]{}
      ++(180:3cm)node(14)[label=right:{=(5)}]{}
      ++(180:3cm)node(15)[label=right:{=(3)}]{};  
 \path (A) edge     (F)
        (B)	edge  (G)
        (C) edge    (H)
        (D)	edge    (I)
        (E) edge    (J);
\end{tikzpicture}
\end{center}

\begin{thm}\label{trisnake}
For any $n\in \mathbb{N}$, $T_n$ is a $k$-geometric mean graph.
\end{thm}
\begin{proof}Let $P_n$ be the path $u_1u_2\dots u_n$ and let $T_n$ be the triangular snake obtained from the path $P_n$ by joining $u_i$ and $u_{i+1}$ to new vertex $v_i$, for $1\leq i\leq n-1$.

Define a function $\psi:V(T_n)\rightarrow \{k,k+1,\dots,k+3n-3=k+q\}$ by

\hspace{2cm}$\psi(u_i)=(k-1)+3i-2,\qquad1\leq i \leq n$

\hspace{2cm}$\psi(v_i)=(k-1)+3i-1, \qquad1\leq i \leq n-1$


For each edge in $T_n$ assign the label as follows
\begin{center}
\begin{tabular}{|c| c| c|}
\hline 
edge $uv$ & label & obtained labels\\\hline&&\\[-8pt]
$u_iu_{i+1},\quad1\leq i \leq n-1$&$\left\lfloor\sqrt{\psi(u)\psi(v)}\right\rfloor$&$\{k+1,k+4,\dots,k+3n-5\}$\\[5pt]\hline&&\\[-8pt]
$u_iv_i,\quad1\leq i \leq n-1$&$\left\lfloor\sqrt{\psi(u)\psi(v)}\right\rfloor$&$\{k,k+3,k+6,\dots,k+3n-6\}$\\[5pt]\hline&&\\[-8pt]
$v_iu_{i+1},\quad1\leq i \leq n-1$&$\left\lceil\sqrt{\psi(u)\psi(v)}\right\rceil$&$\{k+2,k+5,\dots,k+3n-4\}$\\[5pt]\hline
\end{tabular}
\end{center}
Then the set of edge labels is $\{k,k+1,...,k+3n-4=k+q-1\}$. Hence $T_n$ is a $k$-geometric mean graph.
\end{proof}
\begin{ex}
A $5$-geometric mean labeling of $T_5$ is given below.
\end{ex}
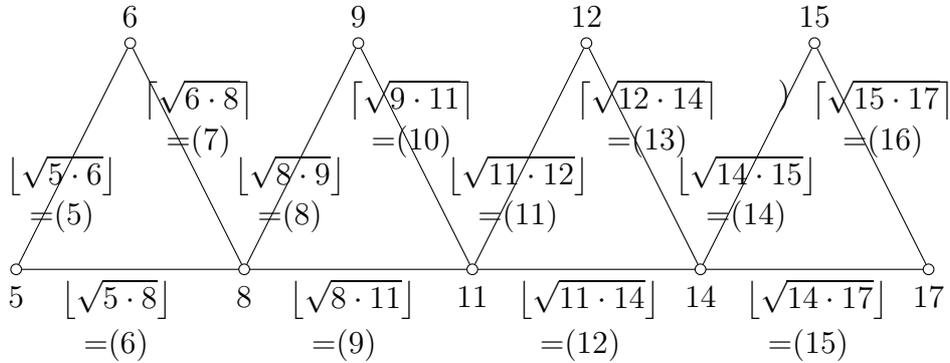
\begin{center}
\begin{tikzpicture}
                            inner sep=0pt]
    \draw (0,0) node[draw,circle,fill=white,minimum size=4pt,
                                inner sep=0pt] (A) [label=below:$5$] {}
        -- ++(0:3cm) node[draw,circle,fill=white,minimum size=4pt,
                                     inner sep=0pt] (B) [label=below:$8$] {}
       --  ++(0:3cm) node[draw,circle,fill=white,minimum size=4pt,
                                     inner sep=0pt] (C) [label=below :$11$] {}
        -- ++(0:3cm) node[draw,circle,fill=white,minimum size=4pt,
                                     inner sep=0pt] (D) [label=below  :$14$] {}
       --  ++(0:3cm) node[draw,circle,fill=white,minimum size=4pt,
                                     inner sep=0pt] (E) [label=below:$17$] {}
        ;
   \draw (1.5,3) node[draw,circle,fill=white,minimum size=4pt,
                                   inner sep=0pt] (F) [label=above:$6$] {}
         ++(0:3cm) node[draw,circle,fill=white,minimum size=4pt,
                                            inner sep=0pt] (G) [label=above:$9$] {} 
          ++(0:3cm) node[draw,circle,fill=white,minimum size=4pt,
                                                     inner sep=0pt] (H) [label=above:$12$]{}  
          ++(0:3cm) node[draw,circle,fill=white,minimum size=4pt, inner sep=0pt](I)[label=above:$15$] {}                                                                                             ;
  \draw (9.3,-0.4) node (11) [label=right:{$\left\lfloor\sqrt{14\cdot 17}\right\rfloor$ }] {}
 ++(180:3cm)node(13)[label=right:{$\left\lfloor\sqrt{11\cdot 14}\right\rfloor$ }]{}
   ++(180:3cm)node(14)[label=right:{$\left\lfloor\sqrt{8\cdot 11}\right\rfloor$ }]{}
   ++(180:3cm)node(15)[label=right:{$\left\lfloor\sqrt{5\cdot 8}\right\rfloor$ }]{};     
 \draw (9.6,-1) node (11) [label=right:{=(15)}] {}
      ++(180:3cm)node(12)[label=right:{=(12)}]{}                                               ++(180:3cm)node(13)[label=right:{=(9)}]{}
      ++(180:3cm)node(14)[label=right:{=(6)}]{};  
\draw (0.6,2.8) node (41) [label=below:{ }] {}
++ (0:1.8cm) node (42) [label=below:{$\left\lceil\sqrt{6 \cdot8}\right\rceil$}] {}
++ (0:1.6cm) node (44) [label=below:{ }] {}
++ (0:1.2cm) node (43) [label=below:{$\left\lceil\sqrt{9 \cdot11}\right\rceil$}] {}
++ (0:1.7cm) node (46) [label=below:{}] {}
++ (0:1.4cm) node (45) [label=below:{$\left\lceil\sqrt{12 \cdot14}\right\rceil$}] {}
++ (0:1.8cm) node (48) [label=below:{ )}] {}
++ (0:1.3cm) node (47) [label=below:{$\left\lceil\sqrt{15 \cdot17}\right\rceil$}] {};
\draw (0.6,2.2) node (31) [label=below:{ }] {}
++ (0:1.8cm) node (32) [label=below:{=(7)}] {}
++ (0:1.2cm) node (33) [label=below:{ }] {}
++ (0:1.6cm) node (34) [label=below:{=(10)}] {}
++ (0:1.4cm) node (35) [label=below:{ }] {}
++ (0:1.7cm) node (36) [label=below:{=(13)}] {}
++ (0:1.3cm) node (37) [label=below:{ }] {}
++ (0:1.8cm) node (38) [label=below:{=(16)}] {};
\draw (0.6,1.8) node (31) [label=below:{$\left\lfloor\sqrt{5\cdot 6}\right\rfloor$}] {}
++ (0:1.8cm) node (32) [label=below:{ }] {}
++ (0:1.2cm) node (33) [label=below:{$\left\lfloor\sqrt{8\cdot 9}\right\rfloor$}] {}
++ (0:1.6cm) node (34) [label=below:{ }] {}
++ (0:1.4cm) node (35) [label=below:{$\left\lfloor\sqrt{11\cdot 12}\right\rfloor$}] {}
++ (0:1.7cm) node (36) [label=below:{ }] {}
++ (0:1.3cm) node (37) [label=below:{$\left\lfloor\sqrt{14\cdot 15}\right\rfloor$}] {}
++ (0:1.8cm) node (38) [label=below:{ }] {};
\draw (0.6,1.2) node (21) [label=below:{=(5)}] {}
++ (0:1.8cm) node (22) [label=below:{ }] {}
++ (0:1.2cm) node (23) [label=below:{=(8)}] {}
++ (0:1.6cm) node (24) [label=below:{ }] {}
++ (0:1.4cm) node (25) [label=below:{=(11)}] {}
++ (0:1.7cm) node (26) [label=below:{ }] {}
++ (0:1.3cm) node (27) [label=below:{=(14)}] {}
++ (0:1.8cm) node (28) [label=below:{ }] {};
 \path (A) edge     (F)
        (B)	edge  (G) edge     (F)
        (C) edge    (H) edge  (G)
        (D)	edge    (I) edge    (H)
        (E) edge    (I);
\end{tikzpicture}
\end{center}
\begin{thm}\label{quadsnake}
For any $n\in \mathbb{N}$, $Q_n$ is a $k$-geometric mean graph.
\end{thm}
\begin{proof}Let $P_n$ be the path $u_1u_2\dots u_n$ and let $Q_n$ be the quadrilateral snake obtained from the path $P_n$ by joining $u_i$ and $u_{i+1}$ to new vertices $v_i,w_i$ in such a way that $u_iv_iw_iv_{i+1}$ is a path, for $1\leq i\leq n-1$.

Define a function $\psi:V(Q_n)\rightarrow \{k,k+1,\dots,k+4n-4=k+q\}$ by

\hspace{2cm}$\psi(u_i)=(k-1)+4i-3,\qquad1\leq i \leq n$

\hspace{2cm}$\psi(v_i)=(k-1)+4i-2, \qquad1\leq i \leq n-1$

\hspace{2cm}$\psi(w_i)=(k-1)+4i-1,\qquad1\leq i \leq n-1$


For each edge in $Q_n$ assign the label as follows
\begin{center}
\begin{tabular}{|c| c| c|}
\hline 
edge $uv$ & label & obtained labels \\\hline&&\\[-8pt]
$u_iu_{i+1},\quad1\leq i \leq n-1$&$\left\lceil\sqrt{\psi(u)\psi(v)}\right\rceil$&$\{k+2,k+6,\dots,k+4n-6\}$\\[5pt]\hline&&\\[-8pt]
$u_iv_i,\quad1\leq i \leq n-1$&$\left\lfloor\sqrt{\psi(u)\psi(v)}\right\rfloor$&$\{k,k+4,k+8,\dots,k+4n-8\}$\\[5pt]\hline&&\\[-8pt]
$v_iw_i,\quad1\leq i \leq n-1$&$\left\lfloor\sqrt{\psi(u)\psi(v)}\right\rfloor$&$\{k+1,k+5,\dots,k+4n-7\}$\\[5pt]\hline&&\\[-8pt]
$w_iu_{i+1},\quad1\leq i \leq n-1$&$\left\lceil\sqrt{\psi(u)\psi(v)}\right\rceil$&$\{k+3,k+7,\dots,k+4n-5\}$\\[5pt]\hline
\end{tabular}
\end{center}
Then the set of edge labels is $\{k,k+1,...,k+4n-5=k+q-1\}$. Hence $Q_n$ is a $k$-geometric mean graph.
\end{proof}
\begin{ex}
A $6$-geometric mean labeling of $Q_4$ is given below.
\end{ex}
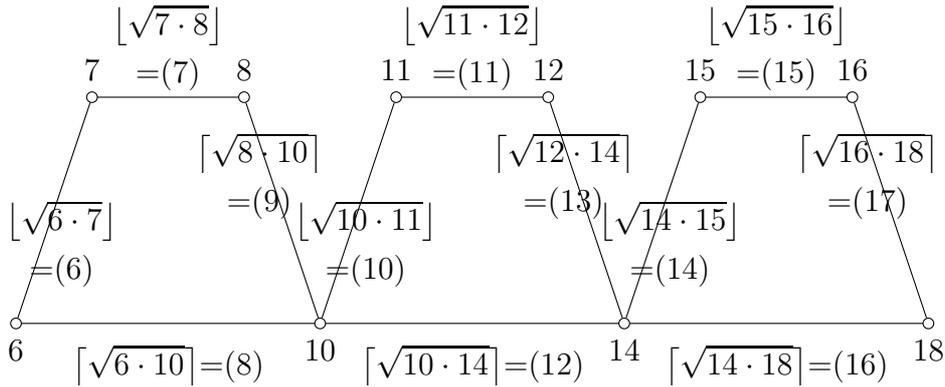
\begin{center}
\begin{tikzpicture}
                            inner sep=0pt]
    \draw (0,0) node[draw,circle,fill=white,minimum size=4pt,
                                inner sep=0pt] (A) [label=below:$6$] {}
        -- ++(0:4cm) node[draw,circle,fill=white,minimum size=4pt,
                                     inner sep=0pt] (B) [label=below:$10$] {}
       --  ++(0:4cm) node[draw,circle,fill=white,minimum size=4pt,
                                     inner sep=0pt] (C) [label=below :$14$] {}
        -- ++(0:4cm) node[draw,circle,fill=white,minimum size=4pt,
                                     inner sep=0pt] (D) [label=below  :$18$] {};
  \draw (1,3) node[draw,circle,fill=white,minimum size=4pt,
                                inner sep=0pt] (v1) [label=above:$7$] {}
         -- ++(0:2cm) node[draw,circle,fill=white,minimum size=4pt,
                                     inner sep=0pt] (w1) [label=above:$8$] {}
         ++(0:2cm) node[draw,circle,fill=white,minimum size=4pt,
                                     inner sep=0pt] (v2) [label=above:$11$] {}
         -- ++(0:2cm) node[draw,circle,fill=white,minimum size=4pt,
                                     inner sep=0pt] (w2) [label=above:$12$] {}
         ++(0:2cm) node[draw,circle,fill=white,minimum size=4pt,
                                     inner sep=0pt] (v3) [label=above:$15$] {} 
        -- ++(0:2cm) node[draw,circle,fill=white,minimum size=4pt,
                                     inner sep=0pt] (w3) [label=above:$16$] {}
        ;
  \draw (2,2.8) node (11) [label=above:{=(7)}] {}
          ++(0:4cm) node (12) [label=above:{=(11)}] {}
         ++(0:4cm) node(13) [label=above:{=(15)}] {};
  \draw (2,3.4) node(21) [label=above:$\left\lfloor\sqrt{7\cdot8}\right\rfloor$] {}
          ++(0:4cm) node(22) [label=above:$\left\lfloor\sqrt{11\cdot12}\right\rfloor$] {}
         ++(0:4cm) node (23) [label=above:$\left\lfloor\sqrt{15\cdot16}\right\rfloor$] {};
 \draw (0.6,0.2) node (41) [label=above:{=(6)}] {}
          ++(0:4cm) node (42) [label=above:{=(10)}] {}
         ++(0:4cm) node(43) [label=above:{=(14)}] {};
  \draw (0.6,0.8) node(51) [label=above:$\left\lfloor\sqrt{6\cdot7}\right\rfloor$] {}
          ++(0:4cm) node(52) [label=above:$\left\lfloor\sqrt{10\cdot11}\right\rfloor$] {}
         ++(0:4cm) node (53) [label=above:$\left\lfloor\sqrt{14\cdot15}\right\rfloor$] {};
 \draw (3.2,1.1) node (41) [label=above:{=(9)}] {}
          ++(0:4cm) node (42) [label=above:{=(13)}] {}
         ++(0:4cm) node(43) [label=above:{=(17)}] {};
  \draw (3.2,1.7) node(51) [label=above:$\left\lceil\sqrt{8\cdot10}\right\rceil$] {}
          ++(0:4cm) node(52) [label=above:$\left\lceil\sqrt{12\cdot14}\right\rceil$] {}
         ++(0:4cm) node (53) [label=above:$\left\lceil\sqrt{16\cdot18}\right\rceil$] {};
  \draw (2,0) node(1) [label=below:{$\left\lceil\sqrt{6 \cdot10}\right\rceil$=(8)}] {}
          ++(0:4cm) node (2) [label=below:{$\left\lceil\sqrt{10 \cdot14}\right\rceil$=(12)}] {}
         ++(0:4cm) node(3) [label=below:{$\left\lceil\sqrt{14 \cdot18}\right\rceil$=(16)}] {};
 \path (A) edge     (v1)
        (B)	edge  (v2) edge(w1)
        (C) edge    (v3) edge(w2)
        (D)	edge    (w3);
\end{tikzpicture}
\end{center}
\subsection{Join of graphs}
From Theorem \ref{comnongeo}, we certain that a graph with more than one component from the set of stars, paths, combs, triangle snakes, or quadrilateral snakes cannot be a geometric mean graph. Here we state a more general result on finite join of graphs.
\begin{thm}A graph $G$ obtained from a finite join of 
cycles and crowns, join with at most one path, comb, triangle snake, or quadrilateral snake, 
is a geometric mean graph. 
\end{thm}
\begin{proof}
Let $G$ be a graph with $n$ components, obtained from a finite join of cycles and crowns, join with at most one path, comb, triangle snake, or quadrilateral snake. 

Let $G_1, G_2,\dots, G_n$ be $n$ pairwise distinct components of $G$, i.e. $G=G_1+G_2+...+G_n$, and let $q_i=|E(G_i)|$, for $1 \leq i \leq n$. If there exists a component of $G$ that is a path, a comb, a triangle snake, or a quadrilateral snake, let that component be $G_n$.
 
Let $q_0=0$. From Theorems \ref{path}, \ref{cycle}, \ref{comb}, \ref{crown}, \ref{trisnake} and \ref{quadsnake},\\ $G_i$ is a  $\left(\displaystyle\sum_{j=0}^{i-1}q_j+1\right)$-geometric mean graph for any $j \in \{1,2,\dots,n\}$. \\Moreover, with the $\left(\displaystyle\sum_{j=0}^{i-1}q_j+1\right)$-geometric mean labeling constructed as in the proof of Theorems \ref{cycle} and \ref{crown} the vertices labels set of $G_i$ are disjoint and hence the result follows. 
\end{proof}
By setting $q_0$ in the proof of the above Theorem to $k-1$, be obtain the following corollary.
\begin{cor}
A graph $G$ obtained from a finite join of 
cycles and crowns, join with at most one path, comb, triangle snake, or quadrilateral snake, 
is a $k$-geometric mean graph. 
\end{cor}
\begin{ex}
A geometric mean labeling of $C_5+C_3+(C_4\odot K_1)+P_4$ is given below.
\end{ex}
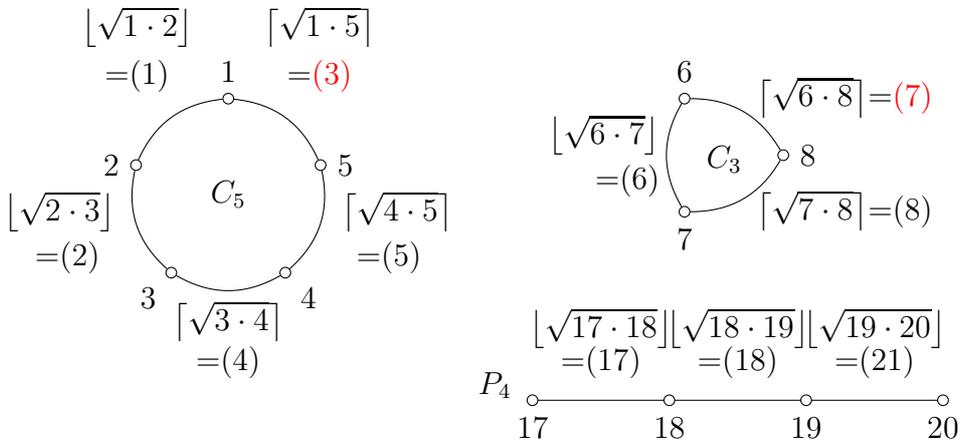
\begin{center}\begin{tikzpicture}
                            inner sep=0pt]
    \draw (0,0) node[draw,circle,fill=white,minimum size=4pt,
                                inner sep=0pt] (A) [label=above:$1$] {}
         ++(216:1.5cm) node[draw,circle,fill=white,minimum size=4pt,
                                     inner sep=0pt] (B) [label=left:$2$] {}
         ++(288:1.5cm) node[draw,circle,fill=white,minimum size=4pt,
                                     inner sep=0pt] (C) [label=below left:$3$] {}
         ++(0:1.5cm) node[draw,circle,fill=white,minimum size=4pt,
                                     inner sep=0pt] (D) [label=below right:$4$] {}
         ++(72:1.5cm) node[draw,circle,fill=white,minimum size=4pt,
                                     inner sep=0pt] (E) [label=right:$5$] {};
	\draw (0,-0.8) node (c5) [label=below:{$C_5$}] {};
\draw (1.2,0.4) node (11) [label=above:{$\left\lceil\sqrt{1\cdot 5}\right\rceil$}] {};
\draw (1.2,-0.2) node (12) [label=above:{=\color{red}(3)}] {};
\draw (-1.2,0.4) node (21) [label=above:{$\left\lfloor\sqrt{1\cdot 2}\right\rfloor$}] {};
\draw (-1.2,-0.2) node (22) [label=above:{=(1)}] {};
\draw (1.2,-1.5) node (31) [label=right:{$\left\lceil\sqrt{4\cdot 5}\right\rceil$}] {};
\draw (1.4,-2.1) node (32) [label=right:{=(5)}] {};
\draw (-1.2,-1.5) node (41) [label=left:{$\left\lfloor\sqrt{2\cdot 3}\right\rfloor$}] {};
\draw (-1.4,-2.1) node (42) [label=left:{=(2)}] {};
\draw (0,-2.4) node (51) [label=below:{$\left\lceil\sqrt{3\cdot 4}\right\rceil$}] {};
\draw (0,-3) node (52) [label=below:{=(4)}] {};

  \path (A) edge [bend right]   (B)
        (B)	edge [bend right]   (C)
        (C) edge [bend right]   (D)
        (D)	edge [bend right]   (E)
        (E) edge [bend right]   (A);

                            inner sep=0pt]
    \draw (6,0) node[draw,circle,fill=white,minimum size=4pt,
                                inner sep=0pt] (A) [label=above:$6$] {}
         ++(270:1.5cm) node[draw,circle,fill=white,minimum size=4pt,
                                     inner sep=0pt] (B) [label=below:$7$] {}
         ++(30:1.5cm) node[draw,circle,fill=white,minimum size=4pt,
                                     inner sep=0pt] (C) [label=right:$8$] {}      
         ++(130:1cm) node  (1) [label=right:{$\left\lceil\sqrt{6\cdot 8}\right\rceil$=\color{red}(7)}] {}
		++(270:1.5cm) node  (3) [label=right:{$\left\lceil\sqrt{7\cdot 8}\right\rceil$=(8)}] {}
		++(150:0.8cm) node  (2) [label=left:{=(6)}] {}
		++(90:0.6cm) node  (2) [label=left:{$\left\lfloor\sqrt{6\cdot 7}\right\rfloor$}] {}
        ;
	\draw (6,-0.8) node (c3) [label=right:{$C_3$}] {};
  \path (A) edge [bend right]   (B)
        (B)	edge [bend right]   (C)
        (C) edge [bend right]   (A);
	\draw (3,-3.8) node (path) [label=right:{$P_4$}] {};
    \draw (4,-4) node[draw,circle,fill=white,minimum size=4pt,
                                inner sep=0pt] (v1) [label= below:$17$] {}
         -- ++(0:1.8cm) node[draw,circle,fill=white,minimum size=4pt,
                                     inner sep=0pt] (v2) [label= below:$18$] {}
         -- ++(0:1.8cm) node[draw,circle,fill=white,minimum size=4pt,
                                     inner sep=0pt] (v3) [label= below:$19$] {}
         -- ++(0:1.8cm) node[draw,circle,fill=white,minimum size=4pt,
                                     inner sep=0pt] (v4) [label=below:$20$] {}
        ;
 	\draw (4.9,-3.6) node (12) [label=above:{$\left\lfloor\sqrt{17\cdot 18}\right\rfloor$}] {};
    	\draw (4.9,-4) node (11) [label=above:{=(17)}] {};   
 	\draw (6.7,-3.6) node (22) [label=above:{$\left\lfloor\sqrt{18\cdot 19}\right\rfloor$}] {};
    	\draw (6.7,-4) node (21) [label=above:{=(18)}] {};   
 	\draw (8.5,-3.6) node (32) [label=above:{$\left\lfloor\sqrt{19\cdot 20}\right\rfloor$}] {};
    	\draw (8.5,-4) node (31) [label=above:{=(21)}] {};   
\end{tikzpicture}\qquad
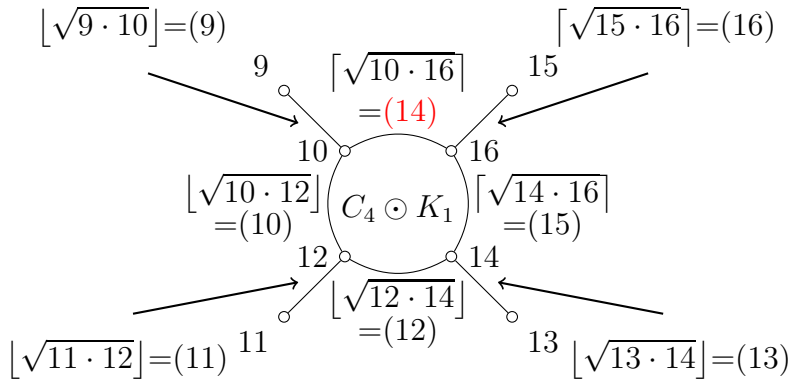
\begin{tikzpicture}
                            inner sep=0pt]
    \draw (0,0) node[draw,circle,fill=white,minimum size=4pt,
                                inner sep=0pt] (v1) [label=above left:$9$] {}
         ++(270:3cm) node[draw,circle,fill=white,minimum size=4pt,
                                     inner sep=0pt] (v2) [label=below left:$11$] {}
         ++(0:3cm) node[draw,circle,fill=white,minimum size=4pt,
                                     inner sep=0pt] (v3) [label=below right:$13$] {}
         ++(90:3cm) node[draw,circle,fill=white,minimum size=4pt,
                                     inner sep=0pt] (v4) [label=above right:$15$] {}
        ;
    \draw (0.8,-0.8) node[draw,circle,fill=white,minimum size=4pt,
                                inner sep=0pt] (u1) [label=left:$10$] {}
         ++(270:1.4cm) node[draw,circle,fill=white,minimum size=4pt,
                                     inner sep=0pt] (u2) [label=left:$12$] {}
         ++(0:1.4cm) node[draw,circle,fill=white,minimum size=4pt,
                                     inner sep=0pt] (u3) [label=right:$14$] {}
         ++(90:1.4cm) node[draw,circle,fill=white,minimum size=4pt,
                                     inner sep=0pt] (u4) [label=right:$16$] {}
        ;
	\draw (1.5,-1.1) node (c5) [label=below:{$C_4\odot K_1$}] {};
    \draw (1.5,-0.8) node (11) [label=above:{=\color{red}(14)}] {};     
	\draw (1.5,-0.3) node (12) [label=above:{$\left\lceil\sqrt{10\cdot 16}\right\rceil$}] {};
    \draw (1.5,-2.7) node (21) [label=below:{=(12)}] {};     
	\draw (1.5,-2.2) node (22) [label=below:{$\left\lfloor\sqrt{12\cdot 14}\right\rfloor$}] {};
    \draw (3.4,-1.3) node (31) [label=below:{=(15)}] {};     
	\draw (3.4,-0.8) node (32) [label=below:{$\left\lceil\sqrt{14\cdot 16}\right\rceil$}] {};
    \draw (-0.4,-1.3) node (41) [label=below:{=(10)}] {};     
	\draw (-0.4,-0.8) node (42) [label=below:{$\left\lfloor\sqrt{10\cdot 12}\right\rfloor$}] {}; 
	\draw (-2,0.3) node (52) [label=above:{$\left\lfloor\sqrt{9\cdot 10}\right\rfloor$=(9)}] {};  
	\draw (5,0.3) node (62) [label=above:{$\left\lceil\sqrt{15\cdot 16}\right\rceil$=(16)}] {};  
	\draw (-2.2,-3) node (72) [label=below:{$\left\lfloor\sqrt{11\cdot 12}\right\rfloor$=(11)}] {};  
	\draw (5.2,-3) node (82) [label=below:{$\left\lfloor\sqrt{13\cdot 14}\right\rfloor$=(13)}] {};
    	\draw (0.4,-0.5) node (p1)  {};   
    	\draw (2.6,-0.5) node (p2)  {};   
    	\draw (0.4,-2.5) node (p3)  {};   
    	\draw (2.6,-2.5) node (p4)  {};    
  \path (u1) edge [bend right]   (u2)
        (u2)	edge [bend right]   (u3)
        (u3) edge [bend right]   (u4)
        (u4)	edge [bend right]   (u1)
	(52) edge [->,thick,
           shorten <=2pt,
           shorten >=2pt] (p1)
	(62) edge[->,thick,
           shorten <=2pt,
           shorten >=2pt] (p2)
	(72) edge [->,thick,
           shorten <=2pt,
           shorten >=2pt](p3)
	(82) edge[->,thick,
           shorten <=2pt,
           shorten >=2pt](p4)
	(u1) edge     (v1)
	(u2) edge     (v2)
	(u3) edge     (v3)
        (u4) edge     (v4);
\end{tikzpicture}
\end{center}
\section{Conclusion}
A new type of labeling, namely $k$-geometric mean labeling, is introduced in this paper as a tool to proof the geometric mean labeling of joins of graphs. We first presented some new geometric mean graph. Then, we define $k$-geometric mean labeling and study the geometric mean labeling on some classes of graphs. Finally, we provided a general result on finite joins of graphs. 

\bibliographystyle{plainnat}
\bibliography{geomean}

\begin{thebibliography}{7}
\providecommand{\natexlab}[1]{#1}
\providecommand{\url}[1]{\texttt{#1}}
\expandafter\ifx\csname urlstyle\endcsname\relax
  \providecommand{\doi}[1]{doi: #1}\else
  \providecommand{\doi}{doi: \begingroup \urlstyle{rm}\Url}\fi

\bibitem[Durai~Baskar and Arockiaraj(2015)]{durai}
A.~Durai~Baskar and S.~Arockiaraj.
\newblock $f-$geometric mean graphs.
\newblock \emph{Application and Applied Mathematics}, 10\penalty0 (2):\penalty0
  937--952, 2015.

\bibitem[Somasundaram et~al.(2011)Somasundaram, Vidhyarani, and Ponraj]{som1}
S.~Somasundaram, P.~Vidhyarani, and R.~Ponraj.
\newblock Geometric mean labeling of graphs.
\newblock \emph{Bulletin of Pure and Applied Sciences}, 30E\penalty0
  (2):\penalty0 153--160, 2011.

\bibitem[Somasundaram et~al.(2012)Somasundaram, Vidhyarani, and Sandhya]{som2}
S.~Somasundaram, P.~Vidhyarani, and S.~Sandhya.
\newblock Some results on geometric mean graphs.
\newblock \emph{International Mathematical Forum}, 7\penalty0 (28):\penalty0
  1381--1391, 2012.

\bibitem[Somasundaram et~al.(2014{\natexlab{a}})Somasundaram, Sandhya, and
  Viji]{som4}
S.~Somasundaram, S.~Sandhya, and S.~P. Viji.
\newblock A note on geometric mean graphs.
\newblock \emph{International Journal of Mathematical Archive}, 5\penalty0
  (10):\penalty0 60--67, 2014{\natexlab{a}}.

\bibitem[Somasundaram et~al.(2014{\natexlab{b}})Somasundaram, Sandhya, and
  Viji]{som5}
S.~Somasundaram, S.~Sandhya, and S.~P. Viji.
\newblock Few results on geometric mean graphs.
\newblock \emph{International Journal of Mathematical Trends and Technology},
  16\penalty0 (1), 2014{\natexlab{b}}.

\bibitem[Somasundaram et~al.(2015{\natexlab{a}})Somasundaram, Sandhya, and
  Viji]{som3}
S.~Somasundaram, S.~Sandhya, and S.~P. Viji.
\newblock Geometric mean labeling of some new disconnected graphs.
\newblock \emph{International Journal of Mathematical Archive}, 6\penalty0
  (6):\penalty0 227--233, 2015{\natexlab{a}}.

\bibitem[Somasundaram et~al.(2015{\natexlab{b}})Somasundaram, Sandhya, and
  Viji]{som6}
S.~Somasundaram, S.~Sandhya, and S.~P. Viji.
\newblock On geometric mean graphs.
\newblock \emph{International Mathematical Forum}, 10\penalty0 (3):\penalty0
  115--125, 2015{\natexlab{b}}.

\end{thebibliography}

\end{document}